\documentclass[11pt]{amsart}
\usepackage{amssymb,latexsym,amsmath,amsfonts}
\usepackage{supertabular}
\usepackage{enumerate}
\usepackage[mathscr]{eucal}
\usepackage{enumitem}

\allowdisplaybreaks

\numberwithin{equation}{section}
\theoremstyle{definition}
\newtheorem{definition}{Definition}[section]

\theoremstyle{remark}
\newtheorem{remark}[definition]{Remark}

\theoremstyle{plain}
\newtheorem{proposition}[definition]{Proposition}
\newtheorem{theorem}[definition]{Theorem}
\newtheorem{lemma}[definition]{Lemma}

\newtheorem{corollary}[definition]{Corollary}

\usepackage[usenames]{color}

\definecolor{Red}{rgb}{1,0,0}
\definecolor{Magenta}{rgb}{0.69,0,0.83}
\definecolor{Green}{rgb}{0.117,0.706,0.314}
\definecolor{Grey}{rgb}{0.8,0.8,0.8}

\newcommand{\koba}{\mathsf{k}}

\newcommand*{\defeq}{\mathrel{\vcenter{\baselineskip0.5ex \lineskiplimit0pt \hbox{\scriptsize.}\hbox{\scriptsize.}}}=}

\newcommand{\C}{\mathbb{C}} 
\newcommand{\R}{\mathbb{R}}

\newcommand{\posint}{\mathbb{Z}_{+}}

\begin{document}
\title{Localization of the Kobayashi distance for any visibility domain}
	
\author{Amar Deep Sarkar}

\address{ADS: Indian Institute of Science Education and Research Kolkata, India}
\email{amar.pdf@iiserkol.ac.in}

	

\thanks{The author is supported by the postdoctoral fellowship of Indian Institute of Science Education and Research Kolkata.}	
\begin{abstract}
In this article, we prove localization results for the Kobayashi distance of Kobayashi hyperbolic domains with local visibility property in $\C^d$, $d \geq 1$. This is done by proving a localization result for the Kobayashi-Royden pseudometric, along with some other results for domains satisfying local weak visibility.
\end{abstract}

\maketitle 
\section{Introduction}
Let $\Omega \subset \C^d$, $d \geq 1$, be a domain and $\koba_{\Omega}: \Omega \times \Omega \longrightarrow \R_{\geq 0}$ and $\kappa: \Omega \times \C^d \longrightarrow \R_{\geq 0}$ denote the Kobayashi pseudodistance and Kobayashi-Royden pseudometric respectively. When $\koba_{\Omega}$ is indeed a distance, we say that $\Omega$ is a Kobayashi hyperbolic domain. The aim of this paper is to obtain localization results for Kobayashi distance under the assumption of weak visibility property (see Definition~\ref{D:Weak_vis}). These localization results could be used to infer about the global geometry from the local geometry and vice versa.\medskip

This kind of  localization result (see Theorem~\ref{T:Main_Thm_1}) is well-known for bounded strongly convex and strongly pseudoconvex domains (cf. \cite{Abate, Balogh_Bonk, Jarnicki_Pflug}); from the work of Zimmer \cite{Zimmer_1, Zimmer_2} it follows that such kind of localization is also true for bounded convexifiable domains of finite type, $\C$-strictly convex domains with $C^{1, \alpha}$ boundary. Recently, Liu--Wang obtained such kind of localization result under the assumption that locally the domains are log-type convex domains, see \cite{Liu_et_al}, and this was substantially generalized by Bracci--Nikolov--Thomas under the assumption that the domains are locally convex and locally has weak visibility property, see \cite{BNT}. In all the above results mentioned, the local convexity or convexifiability assumption is present along with the weak visibility property. However, there exists a large class of domains in $\C^d$ which are not locally convex or locally convexifiable. For example, not all pseudoconvex domains of finite type near a boundary point is locally convexifiable; however, these domains satisfy local visibility property, this follows from the work of Bharali--Zimmer \cite{Bharali_Zimmer_1}, also see \cite{BM}, \cite{BNT}, \cite{CMS}.
\medskip

Now we present our first main result which shows that no assumption of local convexity is required; only local weak visibility assumption is enough to get such localization result. This generalizes \cite[Theorem~1.1]{BNT}. The idea of the proof is inspired by the proof of Theorem~1.1 in \cite{BNT} along with a few important new observations.

\medskip
\begin{theorem}\label{T:Main_Thm_1}
Suppose $\Omega \subset  \C^d$ is a Kobayashi hyperbolic domain and $U$ is an open subset of $\C^d$ such that $U \cap \partial \Omega \neq \emptyset$ and $U \cap \Omega$ is connected. Suppose $( U \cap \Omega, \koba_{U \cap \Omega} )$ has weak visibility property for every pair of distinct points in $ U \cap \partial \Omega$. Then for every open sets $W, W_0$ with $W\cap \Omega \neq \emptyset$, $W \subset \subset W_0 \subset \subset U$ and $\koba_{\Omega}(W_0 \cap \Omega, \Omega \setminus U) > 0$, there exists a constant $C> 0$ which depends only on $U, W, W_0$ such that for every $z, w \in W \cap \Omega$,
\[
\koba_{U \cap \Omega} (z, w) \leq \koba_{\Omega} (z, w) + C.
\]
\end{theorem}
\begin{remark}
If $\Omega$ is a bounded domain, then for every pair $W_0, U$ as in the above theorem, it follows $\koba_{\Omega}(W_0 \cap \Omega, \Omega \setminus U) > 0$. Moreover, {\it (3)} of Proposition~\ref{P:Visibility_Outside_Point} gives a  sufficient condition for $\koba_{\Omega}(W_0 \cap \Omega, \Omega \setminus U) > 0$ when $\Omega$ possibly an unbounded domain. This condition is mentioned in the previous version of this paper without a proof, and later a proof of this also appears in \cite{NOT}.
\end{remark}

Our next main result is the following
\begin{theorem}\label{T:Main_Thm_2}
Suppose $\Omega \subset  \C^d$ is a Kobayashi hyperbolic domain and $U$ is an open subset of $\C^d$ such that $U \cap \partial \Omega \neq \emptyset$. Suppose $( \Omega, \koba_{\Omega} )$ has weak visibility property for every pair of distinct points in $ U \cap \partial \Omega$. Then for every open set $W$ with $W \cap \Omega \neq \emptyset$ and $W \subset \subset U$, there exists a constant $C> 0$ which depends only on $U, W$ such that for every $z, w \in W \cap \Omega$,
\[
\koba_{U \cap \Omega} (z, w) \leq \koba_{\Omega} (z, w) + C.
\]
\end{theorem}

\begin{remark}
Note that in Theorem~\ref{T:Main_Thm_2} the weak visibility assumption is taken with respect to $(\Omega, \koba_{\Omega} )$ whereas in Theorem~\ref{T:Main_Thm_1} the weak visibility is taken with respect to $(U \cap \Omega, \koba_{U \cap \Omega} )$. The second assumption in Theorem~\ref{T:Main_Thm_1} is that $\koba_{\Omega}(W_0 \cap \Omega, \Omega \setminus U) > 0$, however there is no such assumption in Theorem~\ref{T:Main_Thm_2} --- although it is necessary in the proof --- is because of {\it (3)} of Proposition~\ref{P:Visibility_Outside_Point}.
\end{remark}
As a consequence of the above statement, we obtain the following 

\begin{corollary}\label{C:Pseudo_finte_type}
Suppose $\Omega \subset  \C^d$ is a bounded domain with $C^{\infty}$-smooth boundary and of finite D'Angelo type. Let $p \in \partial \Omega$. Then for every neighbourhood $ U$ and $W$ of $p$ with $W \subset \subset U$ there exists a constant $C> 0$ such that  for every $z, w \in W \cap \Omega$,
\[
\koba_{U \cap \Omega} (z, w) \leq \koba_{\Omega} (z, w) + C.
\]
\end{corollary}
\begin{remark}
To the best of our knowledge, such kind of additive localization result for the Kobayashi distance has not been obtained before in this generality. The above corollary could be mentioned for a larger class of domains. We refer the reader to the following papers for such examples of domains: \cite{Bharali_Zimmer_1, Bharali_Zimmer_2, BM, BNT, CMS}. These domains are not necessarily of finite type near every boundary point and in many cases, the
type may not even be defined because of the low regularity of the boundary of the domains. We also emphasize that these domains need not be Cauchy-complete with respect to the Kobayashi distance.
\end{remark}

\subsection*{An application of additive localization}
Next, we show that for bounded domains, multiplicative localization (given below) as a consequence of the additive localization. This shows that additive localization is indeed stronger than multiplicative localization, at least in this case.
\begin{theorem}\label{T:Main_Thm_3}
Suppose $\Omega \subset  \C^d$ is a bounded domain. Suppose $U$ and $W$ are open subsets of $\C^d$ with $W \subset \subset U$, $W \cap \Omega \neq \emptyset$ and $U \cap \Omega$ is connected, and there exists a constant $C_0> 0$ such that for all $z, w \in W \cap \Omega$
\[
\koba_{U \cap \Omega}(z, w) \leq \koba_{\Omega}(z, w) + C_0.
\]
Then there exists a constant  $ C \geq  1$ such that for all $z, w \in W \cap \Omega$,
\[
\koba_{U \cap \Omega}(z, w) \leq C\koba_{\Omega}(z, w).
\]
\end{theorem}
The above theorem follows from a more general result with no boundedness assumption on $\Omega$ but with some other conditions.
\begin{theorem}\label{T:Main_Thm_4}
Suppose $\Omega \subset  \C^d$ is a Kobayashi hyperbolic domain and $W_1, W_2$ are open subsets of $\C^d$ with $W_1 \subset \subset W_2$, $W_1 \cap  \Omega \neq \emptyset$ and $\koba_{\Omega}(W_1 \cap \Omega, \Omega \setminus W_2) > 0$. Suppose $U$ and $W$ are open subsets of with $W \subset \subset W_1 \subset \subset W_2 \subset \subset U$, $W \cap \Omega \neq \emptyset$, $U \cap \Omega$ is connected, and there exists a constant $C_0> 0$ such that for all $z, w \in W \cap \Omega$
\[
\koba_{U \cap \Omega}(z, w) \leq \koba_{\Omega}(z, w) + C_0.
\]
Then there exists a constant  $ C \geq  1$ such that for all $z, w \in W \cap \Omega$,
\[
\koba_{U \cap \Omega}(z, w) \leq C\koba_{\Omega}(z, w).
\]
\end{theorem}
\begin{remark}
The above result is proved using a weaker assumption than that of Theorem~1.4 in \cite{NOT} although the proofs are similar. It can be applied in situations where no form of visibility is present, but additive localization is present. We do not know any non-trivial examples of such domains.
\end{remark}

Very recently, similar localization results are obtained by Nikolov--\"Okten--Thomas \cite{NOT}  under some weaker conditions along with results on local and global visibility. 
\section{Preliminaries}
\subsection{Notations and conventions}

Let $\Omega \subset \C^d$, $d \geq 1$, be a domain, $U \subset \C^d$ an open set, and  $X, Y \subset \C^d $. We denote $\Omega \setminus X$ as the complement of $X \cap \Omega$ in $\Omega$. 
\begin{itemize}
    \item $X \subset \subset U$ means that $X$ is a relatively compact subset of $U$, i.e., $\overline{X} \subset U$ and $\overline{X}$ is a compact set.
    \item For $z \in \Omega$, $\delta_{\Omega}(z) = \inf_{w \in \C^d \setminus \Omega}||z - w||$, where $||z - w||$ denotes the Euclidean distance between $z$ and $w$.
    \item For $z \in \Omega$, $\koba_{\Omega}(z, X) \defeq \inf_{x \in X}\koba_{\Omega}(z, x)$.
    
    \item $\koba_{\Omega}(X, Y) \defeq \inf_{x \in X, y \in Y}\koba_{\Omega}(x, y)$.
    \item Throughout this paper, whenever for an open set $U \subset \C^d$  if $U \cap \partial \Omega \neq \emptyset$, we assume the cardinality of $U \cap \partial \Omega$ is greater than one, i.e., $\#(U \cap \partial \Omega) > 1$, because when $\#(U \cap \partial \Omega) = 1$ the main results of this paper follows trivially. Here $\partial \Omega$ denotes the boundary of the domain $\Omega$.
    \item For a curve $\gamma: I \longrightarrow \Omega$ of an interval $I \subset \R$, by a slight abuse of notation we denote the range of the curve $\gamma$ by $\gamma$ itself.
\end{itemize}

\begin{definition}
For $ \lambda \geq 1$ and $\kappa \geq 0$, a curve $\gamma: I \longrightarrow \Omega$ of an interval $I \subset \R$ is said to be $(\lambda, \kappa)${\bf -quasi-geodesic} if for all $s, t \in I$
\[
\frac{1}{\lambda}|t - s| - \kappa\leq \koba_{\Omega}(\gamma(s), \gamma(t) ) \leq \lambda |s -t| + \kappa,
\]
and when $\lambda =1$ and $\kappa = 0$ $\gamma$ is called a geodesic.
Furthermore, for $ \lambda \geq 1$ and $\kappa \geq 0$, a $(\lambda, \kappa)$-quasi-geodesic $\gamma: I \longrightarrow \Omega$ is said to be $(\lambda, \kappa)${\bf-almost-geodesic} of  $( \Omega, \koba_{\Omega} )$ if,
\begin{itemize}
\item $\gamma$ is an absolutely continuous curve, and
\item $\kappa_{\Omega}(\gamma(t); \gamma^{\prime}(t) ) \leq \lambda$ for almost every $t \in I$, ($\gamma^{\prime}$ exists almost everywhere on $I$ since $\gamma$ is absolutely continuous).
\end{itemize}
\end{definition}

Given any $ \kappa > 0$ and $\Omega \subset \C^d$ a Kobayashi hyperbolic domain (need not be bounded), it is shown in \cite[Proposition~5.3]{Bharali_Zimmer_2} that given any two points $z, w \in \Omega$ there exists a $(1, \kappa)$-almost-geodesic $
\gamma: [0, a] \longrightarrow \Omega$ joining $z$ and $w$, that is, $\gamma(0) = z$ and $\gamma(a) = w$.
Now, we state the definition of visibility and weak visibility.
\begin{definition}\label{D:Weak_vis}
Let $\Omega \subset \C^d$ be a Kobayashi hyperbolic domain. For $\lambda \geq 1$ and $\kappa \geq 0$, we say a pair of points $p, q \in \partial \Omega$, $p \neq q$ has visibility property with respect to $( \Omega, \koba_{\Omega} )$ if there exist neighbourhoods $U$ of $p$ and $V$ of $q$ in $\C^d$ such that $\overline U \cap \overline V \neq \emptyset$ and a compact set $K \subset \Omega$ such that for any $(\lambda, \kappa)$-almost-geodesic $\gamma: [0, a] \longrightarrow \Omega$ with $\gamma(0) \in U \cap \Omega$ and $\gamma(a) \in V \cap \Omega$, $\gamma \cap K \neq \emptyset$. When the above condition is only required for $\lambda = 1$ and $\kappa \geq 0$, we say that the distinct pair $p, q$ has {\bf weak visibility} property with respect to $(\Omega, \koba_{\Omega} )$.
\end{definition}
\section{Proofs of the main results}

The following localization lemma, by L. H. Royden \cite[Lemma~2]{Royden}, whose proof can be found in \cite[Lemma~4]{Ian_Gramham}, is used to prove localization of the Kobayashi distance and to study the relation between local and global visibility and Gromov hyperbolicity in \cite{BNT} and \cite{BGNT} respectively. 
\medskip

\noindent {\bf Royden's Localization Lemma.}
Suppose $\Omega \subset \C^d $ is a Kobayashi hyperbolic domain and $U $ is an open subset of $\C^d$ such that $U \cap \Omega \neq \emptyset$ and $U \cap \Omega $ is connected. 
Then for all $z \in U \cap \Omega$ and $v \in \C^d$,
\begin{equation}
\kappa_{\Omega}(z; v) \leq \kappa_{U \cap \Omega}(z; v) \leq \coth(\koba_{\Omega}(z, \Omega \setminus U)) \kappa_{\Omega}(z; v).
\end{equation}

Next, we prove a lemma that is used to prove the localization of the Kobayashi distance for visibility domains stated above.
\begin{lemma}\label{L:Varied_Loc_Kob_metric}
Suppose $\Omega \subset \C^d $ is a Kobayashi hyperbolic domain and $U $ is an open subset of $\C^d$ such that $U \cap \Omega \neq \emptyset$ and connected. 
Then for every $W \subset \subset U$ with $W \cap \Omega \neq \emptyset$ and $\koba_{\Omega}(W \cap \Omega, \Omega \setminus U) > 0$, there exists $L > 0$ such that for all $z \in W \cap \Omega$ and $v \in \C^d$,
\begin{equation}
\kappa_{U \cap \Omega}(z; v) \leq \left( 1 + L e^{-\koba_{\Omega} (z, \Omega \setminus U) } \right)\kappa_{\Omega}(z; v).
\end{equation}
\end{lemma}
\begin{proof}
We first note the inequality,
\begin{equation}\label{E:Ineq_tanh}
    \tanh(x) \geq 1 - e^{-x}\,\, \forall \,\, x \geq 0.
\end{equation}
This follows from the following formula obtained by algebraic manipulations: for all $x \in \mathbb{R}$,
\[
\tanh(x) - (1 - e^{-x}) = \frac{(1 -e^{-x})^{2}}{e^{x} + e^{-x}} \geq 0.
\]
Now note that, for all $z \in W \cap \Omega$, we have $\coth(\koba_{\Omega}(z, \Omega \setminus U)) \leq \coth(\koba_{\Omega}(W \cap \Omega, \Omega \setminus U)) =: L < +\infty$. By the Royden's Localization Lemma, for all $z \in U\cap \Omega$ and $v \in \C^d$, we have
\[
\tanh(\koba_{\Omega}(z, \Omega \setminus U))\kappa_{U \cap \Omega}(z; v) \leq \kappa_{\Omega}(z; v).
\]
This gives, after applying $\kappa_{ U \cap \Omega}(z; v) \leq \coth(\koba_{\Omega}(z, \Omega \setminus U)) \kappa_{\Omega}(z; v) \leq L \kappa_{\Omega}(z; v) $  for every $z \in W \cap \Omega$ and (\ref{E:Ineq_tanh}), for all $z \in W \cap \Omega$ 
\begin{equation*}
\begin{split}
  \kappa_{U \cap \Omega}(z; v) 
 &\leq (1 - \tanh(\koba_{\Omega}(z, \Omega \setminus U)) ) \kappa_{U \cap \Omega}(z; v) +\kappa_{\Omega}(z; v)\\
 &\leq \left( 1 + L\left(1 - \tanh(\koba_{\Omega}(z, \Omega \setminus U)) \right)  \right) \kappa_{ \Omega}(z; v) \\
 &\leq \left(1 + Le^{- \koba_{\Omega}(z, \Omega \setminus U) } \right)\kappa_{ \Omega}(z; v) \quad [\text{by applying (\ref{E:Ineq_tanh})}].  
    \end{split}
\end{equation*}
This proves the lemma.
\end{proof}
\begin{proposition}\label{P:Visibility_Outside_Point}
Suppose $\Omega \subset  \C^d$ is a Kobayashi hyperbolic domain and $U$ is an open subset of $\C^d$, and $U \cap \partial \Omega \neq \emptyset$. Suppose $(\Omega, \koba_{\Omega} )$ has weak visibility property for  every distinct pair $\{p, q\}$, $ p,q \in U \cap \partial \Omega$. Let $W \subset \subset U$ be an open set with $W \cap \Omega \neq \emptyset$.
Then 
\begin{enumerate}
    \item for any $\xi_2 \in \partial \Omega \setminus  U$ (whenever $\partial \Omega \setminus  U \neq \emptyset$) and $\xi_1 \in U \cap \partial \Omega$, the distinct pair $\{\xi_1, \xi_2\}$ has weak visibility property with respect to $(\Omega, \koba_{\Omega} )$. Consequently, every pair $\{\xi_1, \xi_2\}$, with $\xi_1 \in U \cap \partial \Omega$, $\xi_2 \in  \partial \Omega$ and $\xi_1 \neq \xi_2$, has weak visibility property with respect to $(\Omega, \koba_{\Omega} )$. 
    
    \item Let $V$ is a subset of $\C^d$ with $V \cap \Omega \neq \emptyset$ and $\overline W \cap \overline V = \emptyset$. Then for $\kappa \geq 0$ there is a compact set $K \subset \Omega$ such for every $z \in W \cap \Omega$ and $w \in V \cap \Omega $ and $\gamma$ is a  $(1, \kappa)$-almost-geodesic joining $z$ and $w$, $\gamma \cap K \neq \emptyset$.
    \item For every open sets $W_1, W_2$ with $W \subset \subset W_1 \subset \subset W_2 \subset \subset U$, we have $\koba_{\Omega}(W_1 \cap \Omega, \Omega \setminus W_2) > 0$. Consequently, $\koba_{\Omega}(W \cap \Omega, \Omega \setminus U) > 0$.
\end{enumerate}

\end{proposition}

\begin{proof}
If possible, suppose that {\it (1)} is not true. Then we get a sequence of $(1, \kappa)$-almost-geodesics $\gamma_n: [0, a_n] \longrightarrow \Omega$ such that $\gamma_n(0)  \to \xi_1$ and  $\gamma_n(a_n)  \to \xi_2$ and $ \lim_{n \to \infty} \sup_{z \in \gamma_n \cap B(0, R)}\delta_{\Omega}(z) = 0 $, for all $R> R_0$, where we choose $R_0 > 0$ such that $\gamma_n \cap  B(0, R_0) \neq \emptyset $ for all $n \in \posint$. Since $U$ is an open subset of $\C^d$ and $\xi_1 \in U$, there exists $0 < r < ||\xi_1 - \xi_2||/4$ such that the open ball of radius $2r$ centred at $\xi_1$, $B(\xi_1, 2r) \subset U$. Without loss of generality (by passing to a subsequence if necessary), we may assume that $\gamma_n(0) \in B(\xi_1, r/2)$ and $\gamma_n(a_n) \in B(\xi_2, r/2)$ for all $n \in \posint$. Now, let
\[
t_n \defeq \inf\{t \in [0, a_n] : \gamma_n(t) \in \partial B(\xi_1, r) \cap \Omega \}.
\]
Again, by passing to a subsequence, we may assume that $ \gamma_n(t_n) \to \xi_0 \in (\partial B(\xi_1, r) \cap \partial \Omega) \subset U \cap \partial \Omega$. Now, note that the  sequence of $(1, \kappa)$-almost-geodesics defined as
\[
\sigma_n \defeq \gamma_n|_{[0, t_n]}
\]
shows that the distinct pair $\{ \xi_1, \xi_0 \} \subset  U \cap \partial \Omega $ does not have weak visibility property with respect to $(\Omega, \koba_{\Omega} )$. This gives a contradiction to our assumption. This proves {\it (1)}.

Next, we shall prove {\it (2)}. If possible, to get a contradiction, suppose that {\it (2)} does not hold. Then there exist $\kappa \geq 0$, sequences of points $z_n \in W \cap \Omega$ and $w_n \in V \cap \Omega$ and $(1, \kappa)-$almost-geodesic $\gamma_n$ of $(\Omega, \koba_{\Omega} )$ joining $z_n$ and $w_n$ for all $n \in \posint$, such that $\lim_{n \to \infty} \sup_{z \in \gamma_n \cap B(0, R)}\delta_{\Omega}(z) = 0 $, for all $R> R_1$, where we choose  $R_1 > 0$ such that $ \overline W \subset   B(0, R_1)  $. By passing to a subsequence, we may assume that $z_n \to \xi'_1 \in \overline W \cap \partial \Omega \subset U \cap \partial \Omega $ as $n \to \infty$. 
Since $\overline {W \cap \Omega}  \cap \overline{V} = \emptyset$, let $0 < r_1 < \frac{||x - y||}{4}$ for all $x \in W \cap \Omega $ and $y \in V$. Assume, by passing to a subsequence, $z_n \in B(\xi'_1, r_1/2) $ for all $n \in \posint$, and define
\[
t'_n \defeq \inf\{t \in [0, a_n] : \gamma_n(t) \in \partial B(\xi'_1, r_1) \cap \Omega \}.
\]
Without loss of generality, we may assume that $\gamma_n(t'_n) \to \xi'_2 \in \partial B(\xi'_1, r_1) \cap \partial \Omega$ as $n \to \infty$. By construction,
 $\xi'_1 \neq \xi'_2$. Now, note that the  sequence of $(1, \kappa)$-almost-geodesics defined as
\[
\sigma'_n \defeq \gamma_n|_{[0, t'_n]}
\]
shows that the distinct pair $\{ \xi'_1, \xi'_2 \} \subset  \partial \Omega $ with $\xi'_1 \in U \cap \partial \Omega$ does not have weak visibility property with respect to $(\Omega, \koba_{\Omega} )$. This gives a contradiction to {\it (1)}. Hence, {\it (2)} holds true.

To prove {\it (3)} by contradiction if possible let $\koba_{\Omega}(W_1 \cap \Omega, \Omega \setminus W_2) = 0$. Then there exist sequences of points $z_n \in W_1 \cap \Omega$ and $w_n \in \Omega \setminus W_2$ such that $\koba_{\Omega}(z_n, w_n) \to 0$ as $n \to \infty$. Let $\gamma_n: [0, a_n] \longrightarrow \Omega$ is a sequence of $(1, 1/n)$-almost-geodesic such that $\gamma_n(0) = z_n \in W_1 \cap \Omega$ and $\gamma_n(a_n) = w_n \in \Omega \setminus W_2$ for all $n \in \posint$. Since $W_1$ is a relatively compact subset of U, and $\overline {W_1} \cap \overline{\Omega \setminus W_2} = \emptyset$, by {\it (2)}, for $\kappa \geq 1 $ there exists a compact set $K \subset \Omega$ such that for every $z \in W_1 \cap \Omega$ and $w \in \Omega \setminus W_2 $ if $\gamma$ is a $(1, \kappa)$-almost-geodesic joining $z$ and $w$, then $\gamma \cap K \neq \emptyset$. Hence, 
    $\gamma_n \cap K \neq \emptyset$ for all $n \in \posint$. Let $o_n \in \gamma_n \cap K$ and if necessary by passing to a subsequence, assume that $o_n \to o \in K$. This gives, using the fact that $\gamma_n$ is a $(1, 1/n)$-almost-geodesic passing through $o_n$,
    \[
    \koba_{\Omega}(z_n, o_n) + \koba_{\Omega}(o_n, w_n) - 3/n \leq \koba_{\Omega}(z_n, w_n). 
    \]
    This implies as $n \to \infty$, $\koba_{\Omega}(z_n, o_n) \to 0$ and $\koba_{\Omega}( o_n, w_n) \to 0$. Since $\Omega$ is a Kobayashi hyperbolic, it follows $z_n \to o$ and $w_n \to o$. This is a contradiction because $z_n \in W_1 \cap \Omega$, $w_n \in \Omega \setminus W_2$ for all $n \in \posint$, and $\overline{W_1 \cap \Omega} \cap \overline{\Omega \setminus W_2} = \emptyset$. Hence $\koba_{\Omega}(W_1 \cap \Omega, \Omega \setminus W_2) >  0$. 
    
    Moreover, using the fact that if $X, Y \subset \Omega$ and $A \subset X$ and $B \subset Y$, then by definition $\koba_{\Omega}( A, B) \geq \koba_{\Omega}(X, Y)$, we have $\koba_{\Omega}(W \cap \Omega, \Omega \setminus U) \geq \koba_{\Omega}(W_1 \cap \Omega, \Omega \setminus W_2) >  0$. This completes the proof of {\it (3)}.
\end{proof}

\begin{lemma}\label{L:Comple_dist_Compar}
Suppose $\Omega \subset  \C^d$ is a Kobayashi hyperbolic domain and $U$ is an open subset of $\C^d$ such that  $U \cap \partial \Omega \neq \emptyset$. Suppose $( \Omega, \koba_{\Omega} )$  has weak visibility property for every pair of distinct points $\{p, q\}$, $p, q \in U \cap \partial \Omega$.
Then for every open sets $W_1, W_2$ with $W_1 \subset \subset W_2 \subset \subset U$, $W_1 \cap \Omega \neq \emptyset$ and $o \in \Omega$ there exists $L > 0$ such that for all $z \in W_1 \cap \Omega$ 
\[
\koba_{\Omega} ( z, o) \leq \koba_{\Omega} (z, \Omega \setminus W_2) + L.
\]
\end{lemma}
\begin{proof}
Let $z \in W_1 \cap \Omega$ and $w \in \Omega \setminus W_2$.
Let $\kappa > 0$ and $\gamma$ is a $(1, \kappa)$-almost-geodesic joining $z$ and $w$. Then, by {\it (2)} of Proposition~\ref{P:Visibility_Outside_Point}, there exists a compact subset $K \subset \Omega$ such that $\gamma \cap K \neq \emptyset$. Let $o_w \in \gamma \cap K$, then we have (using the fact that $\gamma$ is $(1, \kappa)$-almost-geodesic)
\begin{equation*}
    \begin{split}
     \koba_{\Omega}(z, o) 
     &\leq \koba_{\Omega}(z, o_w) + \koba_{\Omega}(o_w, o) \\
     & \leq \koba_{\Omega}(z, w) + 2\kappa + \sup_{y \in K}\koba_{\Omega}(y, o).
     \end{split}
\end{equation*}
Since $w \in \Omega \setminus W_2$ is arbitrary, by taking $L \defeq 2\kappa + \sup_{y \in K}\koba_{\Omega}(y, o) $, we have,
\[
\koba_{\Omega} ( z, o) \leq \koba_{\Omega} (z, \Omega \setminus W_2) + L.
\]
\end{proof}
As a corollary of the above lemma, we have a similar local result.

\begin{lemma}\label{L:Comple_dist_Compar_local}
Suppose $\Omega \subset  \C^d$ is a Kobayashi hyperbolic domain and $U$ is an open subset of $\C^d$ such that $U \cap \partial \Omega \neq \emptyset$ and $U \cap \Omega $ is connected. Suppose $( U \cap \Omega, \koba_{U \cap \Omega} )$ has weak visibility property for every pair of distinct points $\{p, q\}$, $p, q \in U \cap \partial \Omega$.   
Then for every open sets $W_1, W_2$ with $W_1 \subset \subset W_2 \subset \subset U$, $W_1 \cap \Omega \neq \emptyset$ and $o \in U \cap\Omega$ there exists $L > 0$ such that for all $z \in W_1 \cap \Omega$ 
\[
\koba_{U \cap \Omega} ( z, o) \leq \koba_{U \cap \Omega} (z, U \cap \Omega \setminus W_2) + L.
\]
\end{lemma}
\begin{proof}
The proof follows replacing $\Omega$ by $U \cap \Omega$ in Lemma~\ref{L:Comple_dist_Compar}. 
\end{proof}

\begin{lemma}\label{L:One_kappa_geodesic}
Suppose $\Omega \subset  \C^d$ is a Kobayashi hyperbolic domain and $U$ is an open subset of $\C^d$ such that $U \cap \partial \Omega \neq \emptyset$ and $U \cap \Omega $ is connected. Suppose  $( U \cap \Omega, \koba_{U \cap \Omega} )$ has weak visibility property for every pair of distinct points $\{p, q\}$, $p, q \in U \cap \partial \Omega$. Let $W, W_1$ be open sets with $W \subset \subset W_1 \subset \subset U $, $W \cap \Omega \neq \emptyset$ and $\koba_{\Omega}(W_1 \cap \Omega, \Omega \setminus U) > 0$, and suppose that $z, w \in W \cap \Omega $ and  $\gamma$ is a $(1, \kappa)$-almost-geodesic of $(\Omega, \koba_{\Omega})$ joining $z$ and $w$ such that $\gamma \subset W$. Then there exists a $\kappa_0 > 0$ such that $\gamma$ is $(1, \kappa_0)$-quasi-geodesic of $(U \cap \Omega, \koba_{U \cap \Omega} )$.
\end{lemma}
\begin{proof}
By Lemma~\ref{L:Varied_Loc_Kob_metric}, there exists $L > 0$ which depends on $U, W_1$ such that for all $z \in W_1 \cap \Omega$ and $v \in \C^d$, 
\begin{equation*}
\kappa_{U \cap \Omega}(z; v) \leq \left( 1 + L e^{-\koba_{\Omega} (z, \Omega \setminus U) } \right)\kappa_{\Omega}(z; v).
\end{equation*}
Since $\koba_{\Omega} (z, \Omega \setminus W_1) \leq \koba_{\Omega} (z, \Omega \setminus U)$, we have
\begin{equation}\label{E:Lemma_use_Roy_loc}
\kappa_{U \cap \Omega}(z; v) \leq \left( 1 + L e^{-\koba_{\Omega} (z, \Omega \setminus W_1) } \right)\kappa_{\Omega}(z; v).
\end{equation}

Suppose that $\gamma: [0, a] \longrightarrow \Omega$, $a \geq 0$.
Since $\gamma \subset W$, from above inequality and by \cite[Theorem~3.1]{Venturini}, for every $s_1, s_2 \in [0, a]$, we have
\begin{equation}\label{E:One_kappa_geo_ineq}
\begin{split}
\koba_{U \cap \Omega}(\gamma(s_1), \gamma(s_2)) 
& \leq \int_{s_1}^{s_2} \kappa_{U \cap \Omega}(\gamma(t); \gamma^{\prime}(t)) dt\\
& \leq \int_{s_1}^{s_2} \left( 1 + L e^{-\koba_{\Omega} (\gamma(t), \Omega \setminus W_1) } \right)\kappa_{\Omega}(\gamma(t); \gamma^{\prime}(t)) dt\\
&\leq  |s_2 - s_1| + L \int_{s_1}^{s_2} e^{-\koba_{\Omega} (\gamma(t), \Omega \setminus W_1) } dt  \\ 
& \leq \koba_{\Omega}(\gamma(s_1), \gamma(s_2)) + \kappa + L \int_{0}^{a} e^{-\koba_{\Omega} (\gamma(t), \Omega \setminus W_1) } dt.
\end{split}
\end{equation}
{\bf Claim.}
There exist $C_0, C_1 > 0$ such that for all $z \in W \cap \Omega$ 
\[
\koba_{U  \cap \Omega}(z, U \cap \Omega \setminus W_1) \leq C_0 \koba_{\Omega}(z, \Omega \setminus W_1) + C_1. 
\]
Assuming the claim and deferring the proof, first note that, for $o \in U \cap \Omega$, by Lemma~\ref{L:Comple_dist_Compar_local}, there exists $L_0 > 0$ which depends on $o, W, W_1, U$ such that for all $z \in W \cap \Omega$, 
\[
\koba_{U \cap \Omega} ( z, o) \leq \koba_{U \cap \Omega} (z, U \cap \Omega \setminus W_1) + L_0.
\]
Next, let $t_0 \in [0, a]$ such that for all $t \in [0, a]$,
\[
\koba_{\Omega}(\gamma(t_0), \Omega \setminus W_1) \leq \koba_{\Omega}(\gamma(t), \Omega \setminus W_1).
\]
Now, applying the above inequalities in the following computation, we obtain, for all $t \in [0, a]$
\begin{equation*}
   \begin{split}
       |t - t_0| - \kappa 
       &\leq \koba_{\Omega}(\gamma(t), \gamma(t_0))\\ 
       & \leq \koba_{\Omega}(\gamma(t), o) + \koba_{\Omega}(\gamma(t_0), o)\\
       & \leq \koba_{U \cap \Omega}(\gamma(t), o) + \koba_{U \cap \Omega}(\gamma(t_0), o)\\
       &\leq \koba_{U \cap \Omega} (\gamma(t), U \cap \Omega \setminus W_1) + \koba_{U \cap \Omega} (\gamma(t_0), U \cap \Omega \setminus W_1) + 2L_0\\
       &\leq 2\koba_{U \cap \Omega} (\gamma(t), U \cap \Omega \setminus W_1) + 2L_0\\
   \end{split} 
\end{equation*}
If we apply the claim to the above inequality, we get
\begin{equation*}
   \begin{split}
       |t - t_0| - \kappa 
       &\leq 2\koba_{U \cap \Omega} (\gamma(t), U \cap \Omega \setminus W_1) + 2L_0\\
       & \leq 2C_0 \koba_{\Omega}(\gamma(t), \Omega \setminus W_1) + 2C_1 + 2 L_0. 
   \end{split} 
\end{equation*}
This gives, for all $ t \in [0, a]$
\[
-\koba_{\Omega}(\gamma(t), \Omega \setminus W_1) \leq \frac{ -|t - t_0|}{2C_0} + \frac{2C_1 + 2L_0 + \kappa}{2C_0}.
\]
After using the above inequality in inequality~(\ref{E:One_kappa_geo_ineq}), we obtain for every $s_1, s_2 \in [0, a]$,
\begin{equation}
\begin{split}
\koba_{U \cap \Omega}(\gamma(s_1), \gamma(s_2))
& \leq \koba_{\Omega}(\gamma(s_1), \gamma(s_2)) + \kappa + LC_3 \int_{0}^{a} e^{-\frac{|t - t_0|}{2C_0} } dt\\
&\leq \koba_{\Omega}(\gamma(s_1), \gamma(s_2)) + \kappa + LC_3 \int_{0}^{+\infty} e^{-\frac{|t - t_0|}{2C_0} } dt\\
&\leq \koba_{\Omega}(\gamma(s_1), \gamma(s_2)) + \kappa + 4LC_3 C_0
\end{split}
\end{equation}
where $C_3 = e^{\frac{2C_1 + 2L_0 + \kappa}{2C_0}}$, and the last inequality follows from the following bound:
\begin{equation}
 \int_{0}^{+\infty} e^{-\frac{|t - t_0|}{2C_0} } dt \leq 4C_0.
\end{equation}
This shows that $\gamma$ is a $(1, \kappa_0)$-quasi-geodesic for $\kappa_0 = \kappa + 4LC_3C_0$, given that the claim is true (note the fact that $\kappa_0$ depends only on $W,W_1,U$ and $\kappa$). Now, we give a proof of the claim.
\begin{proof}[Proof of the claim.]
By the Royden's Localization Lemma and setting $ \coth{\koba_{\Omega}(W_1 \cap \Omega, \Omega \setminus U)} := C_0 < +\infty$, for all $z \in W_1 \cap \Omega $ and $v \in \C^d$,
\[
\kappa_{U  \cap \Omega}(z; v) \leq \coth{\koba_{\Omega}(z, \Omega \setminus U)} \kappa_{\Omega}(z;v) \leq \coth{\koba_{\Omega}(W_1 \cap \Omega, \Omega \setminus U)} \kappa_{\Omega}(z;v) = C_0 \kappa_{\Omega}(z;v).
\]
Now, suppose that $w_n \in \Omega \setminus W_1$ such that $\lim_{n \to \infty} \koba_{\Omega}(z, w_n) = \koba_{\Omega}(z, \Omega \setminus W_1) $ and $\gamma_n : [0, a_n] \longrightarrow \Omega$ is sequence of $(1, \kappa)$-almost-geodesic of $(\Omega, \koba_{\Omega} )$ joining $z$ and $w_n$, for some $\kappa > 0$. 
\[
t_n = \sup \{t : t \in [0, a_n]\,\, \text{such that}\,\, \gamma_n(s) \in W_1 \cap \Omega \,\forall s \in [0, t]\},
\]

By definition, $\gamma_n([0, t_n)) \subset W_1 \cap \Omega$. Then,
\begin{equation*}
    \begin{split}
     \koba_{U \cap \Omega}(z, U \cap \Omega \setminus W_1) 
     &\leq \koba_{U \cap \Omega}(z, \gamma_n(t_n) )\\
     & \leq \int_{0}^{t_n} \kappa_{U  \cap \Omega}(\gamma_n(t); \gamma_n^{\prime}(t)) dt\\
     & \leq C_0 \int_{0}^{t_n} \kappa_{ \Omega}(\gamma_n(t); \gamma_n^{\prime}(t)) dt\\
     &\leq C_0 \int_{0}^{a_n} \kappa_{ \Omega}(\gamma_n(t); \gamma_n^{\prime}(t)) dt\\
     &\leq C_0 \koba_{\Omega}(z, w_n) + C_0 \kappa.\\
    \end{split}
\end{equation*}
The third inequality follows because $\kappa_{U  \cap \Omega}(\gamma_n(t); \gamma_n^{\prime}(t)) \leq C_0 \kappa_{\Omega}(\gamma_n(t); \gamma_n^{\prime}(t)) $ a.e. $t \in [0, a_n] $.
Taking limit as $n \to \infty$, we get
\[
\koba_{U \cap \Omega}(z, U \cap \Omega \setminus W_1) \leq C_0\koba_{\Omega}(z, \Omega \setminus W_1) + C_0 \kappa.
\]
This proves the claim. Hence the proof is complete.
\end{proof}
\end{proof}

\begin{lemma}\label{L:One_kappa_geodesic_Omega}
Suppose $\Omega \subset  \C^d$ is a domain and $U$ is an open subset of $\C^d$ such that $U \cap \partial \Omega \neq \emptyset$. Suppose  $(\Omega, \koba_{\Omega} )$ has weak visibility property for every pair of distinct points $\{p, q\}$, $p, q \in U \cap \partial \Omega$. Let $W \subset \subset U $ be an open set with $W \cap \Omega \neq \emptyset$, and suppose that $z, w \in W \cap \Omega $ and  $\gamma$ is a $(1, \kappa)$-almost-geodesic of $(\Omega, \koba_{\Omega})$ joining $z$ and $w$ such that $\gamma \subset W$. Then there exists a $\kappa_0 > 0$ such that $\gamma$ is $(1, \kappa_0)$-quasi-geodesic of $(U \cap \Omega, \koba_{U \cap \Omega} )$.
\end{lemma}
\begin{proof}
Let $W_1 \subset \C^d$ be an open set such that $W \subset \subset W_1 \subset \subset U$. By {\it (3)} of Proposition~\ref{P:Visibility_Outside_Point}, $\koba_{\Omega}(W_1 \cap \Omega, \Omega \setminus U) >0 $.

Therefore, (as observed at the start of the proof of Lemma~\ref{L:One_kappa_geodesic}),
by Lemma~\ref{L:Varied_Loc_Kob_metric}, there exists $L > 0$ which depends on $U, W_1$ such that for all $z \in W_1 \cap \Omega$ and $v \in \C^d$,
\begin{equation*}
\kappa_{U \cap \Omega}(z; v) \leq \left( 1 + L e^{-\koba_{\Omega} (z, \Omega \setminus W_1) } \right)\kappa_{\Omega}(z; v).
\end{equation*}
Suppose that $\gamma: [0, a] \longrightarrow \Omega$. Then by doing similar computations as of (\ref{E:One_kappa_geo_ineq}) of Lemma~ \ref{L:One_kappa_geodesic}, for every $s_1, s_2 \in [0, a]$, we have
\begin{equation}\label{E:One_kappa_geo_ineq_2}
\begin{split}
\koba_{U \cap \Omega}(\gamma(s_1), \gamma(s_2)) 
& \leq \koba_{\Omega}(\gamma(s_1), \gamma(s_2)) + \kappa + L \int_{0}^{a} e^{-\koba_{\Omega} (\gamma(t), \Omega \setminus W_1) } dt.
\end{split}
\end{equation}

Now, by Lemma~\ref{L:Comple_dist_Compar}, there exists $L_0 > 0$ which depends on $o, W, W_1, U$ such that for all $z \in W \cap \Omega$, 
\[
\koba_{\Omega} ( z, o) \leq \koba_{\Omega} (z, \Omega \setminus W_1) + L_0.
\]
Next, let $t_0 \in [0, a]$ such that for all $t \in [0, a]$,
\[
\koba_{\Omega}(\gamma(t_0), \Omega \setminus W_1) \leq \koba_{\Omega}(\gamma(t), \Omega \setminus W_1).
\]
Since $\gamma$ is a $(1, \kappa)$-almost-geodesic of $(\Omega, \koba_{\Omega} )$, we have by a simple computation (using the above inequalities and the triangle inequality), for all $t \in [0, a]$
\begin{equation*}
   \begin{split}
    |t - t_0| - \kappa 
    &\leq 2\koba_{\Omega} (\gamma(t), \Omega \setminus W_1) + 2L_0.
   \end{split} 
\end{equation*}
This gives, for all $ t \in [0, a]$
\[
-\koba_{\Omega}(\gamma(t), \Omega \setminus W_1) \leq \frac{ -|t - t_0|}{2} + \frac{2L_0 + \kappa}{2}.
\]
After using the above inequality in inequality~(\ref{E:One_kappa_geo_ineq_2}), we obtain for every $s_1, s_2 \in [0, a]$,
\begin{equation}
\begin{split}
\koba_{U \cap \Omega}(\gamma(s_1), \gamma(s_2))
& \leq \koba_{\Omega}(\gamma(s_1), \gamma(s_2)) + \kappa + LC_3 \int_{0}^{a} e^{-\frac{|t - t_0|}{2} } dt\\
&\leq \koba_{\Omega}(\gamma(s_1), \gamma(s_2)) + \kappa + 4LC_3 C_0
\end{split}
\end{equation}
where $C_3 = e^{\frac{2L_0 + \kappa}{2}}$, and the last inequality follows from the following bound:
\begin{equation}
 \int_{0}^{+\infty} e^{-\frac{|t - t_0|}{2} } dt \leq 4.
\end{equation}
This shows that $\gamma$ is a $(1, \kappa_0)$-quasi-geodesic for $\kappa_0 = \kappa + 4LC_3$ (note the fact that $\kappa_0$ depends only on $W,W_1,U$ and $\kappa$). Hence the proof is complete.
\end{proof}

\begin{proof}[Proof of Theorem~\ref{T:Main_Thm_1}]
Let $W_1, W_2$ be open sets such that $W \subset \subset W_1 \subset \subset W_2 \subset \subset W_0$. Let $z, w \in W \cap \Omega$. Let $ \kappa > 0$ and $\gamma : [0, a] \longrightarrow \Omega$ be a $(1, \kappa)$-almost-geodesic of $(\Omega, \koba_{\Omega} )$ joining $z$ and $w$. 

\noindent
{\bf Case 1.} $\gamma \subset W_2 \cap \Omega $. 

Since $\koba_{\Omega}(W_0 \cap \Omega, \Omega \setminus U) > 0$, by replacing $W$ and $W_1$  by $W_2$ and $W_0$ respectively in Lemma~\ref{L:One_kappa_geodesic}, there exists $\kappa_0 > 0$ which depends on $W_2, W_0, U$ and $\kappa$ such that $\gamma$ is a $(1, \kappa_0)$-quasi-geodesic of $(U \cap \Omega, \koba_{U \cap \Omega} )$.
This gives, using the fact that $\gamma$ is a $(1, \kappa_0)$-quasi-geodesic of $(U \cap \Omega, \koba_{U \cap \Omega} )$ and $\gamma$ is a $(1, \kappa)$-almost-geodesic of $(\Omega, \koba_{\Omega} )$,
\[
\koba_{U \cap \Omega }(z, w) \leq a + \kappa_0 \leq \koba_{ \Omega }(z, w) + \kappa + \kappa_0.
\]
\noindent
{\bf Case 2.} $\gamma \nsubseteq W_2 \cap \Omega $.
Now, set
\[
s_0 = \sup \{t : t \in [0, a]\,\, \text{such that}\,\, \gamma(s) \in W_1 \,\forall s \in [0, t]\},
\]
\[
t_0= \inf \{t : t \in [0, a]\,\, \text{such that}\,\, \gamma(s) \in W_1 \,\forall s \in [t, a]\}.
\]
Note that by definition, $\gamma([0, s_0]) \subset W_2$ and $\gamma([t_0, a]) \subset W_2$. Hence $\gamma|_{[0, s_0]}$ and $ \gamma|_{[t_0, a]}$ are $(1, \kappa_0)$-quasi-geodesics of $(U \cap \Omega, \koba_{U \cap \Omega} )$ (as observed in Case 1). 

Now, let $\sigma^1 : [0, b^1] \longrightarrow U \cap \Omega $ and $\sigma^2: [0, b^2] \longrightarrow U \cap \Omega $ be  $(1, \kappa)$-almost-geodesics of $(U \cap \Omega, \koba_{U \cap \Omega} )$ joining $z$ and $\gamma(s_0)$, and $w$ and $\gamma(t_0)$ respectively. Let $U' \subset \subset U$ such that $\#(U' \cap \partial \Omega) > 1$ and $W_0 \subset \subset U'$. Since $U' \cap \partial (U \cap \Omega) = U'\cap \partial \Omega$, every pair of distinct boundary points in $U'\cap \partial \Omega$ has weak visibility property with respect to $(U \cap \Omega, \koba_{\Omega} )$.
Then, by {\it (2)} of Proposition ~\ref{P:Visibility_Outside_Point} by replacing the domain $\Omega$ by $ U\cap \Omega$, $U$ by $U'$, $W$ by $W$ and $V$ by $ (U \cap \Omega) \setminus W_1$, there exists a compact set $K \subset U \cap \Omega$ such that $\sigma^1 \cap K$ and $\sigma^2 \cap K$ are non-empty sets. Let $ o^1 \in \sigma^1 \cap K$ and $o^2 \in \sigma^2 \cap K$. 

Now, using the fact that $\sigma^1 $ and $\sigma^2 $ are $(1, \kappa)$-almost-geodesics of $(U \cap \Omega, \koba_{U \cap \Omega} )$ and the triangle inequality, we have the following  inequality
\begin{equation}\label{E:C_inequality_1}
\begin{split}
    \koba_{U \cap \Omega}(z, \gamma(s_0))
    &\geq \koba_{U \cap \Omega}(z, o^1) + \koba_{U \cap \Omega}(o^1, \gamma(s_0))  - 3\kappa\\
    &\geq \koba_{U \cap \Omega}(z, o) - \koba_{U \cap \Omega}(o, o^1) + \koba_{U \cap \Omega}(o, \gamma(s_0)) - \koba_{U \cap \Omega}(o, o^1)  - 3\kappa\\
    &\geq \koba_{U \cap \Omega}(z, o) + \koba_{U \cap \Omega}(o, \gamma(s_0)) - 2 \sup_{y \in K }\koba_{U \cap \Omega}(o, y)  - 3\kappa,
\end{split}    
\end{equation}
and similarly
\begin{equation}\label{E:C_inequality_2}
\begin{split}
    \koba_{U \cap \Omega}(w, \gamma(t_0))
    &\geq \koba_{U \cap \Omega}(w, o) + \koba_{U \cap \Omega}(o, \gamma(t_0)) - 2 \sup_{y \in K }\koba_{U \cap \Omega}(o, y)  - 3\kappa.
\end{split}    
\end{equation}
Let $C_0 \defeq 2 \sup_{y \in K }\koba_{U \cap \Omega}(o, y)  + 3\kappa$.

Now, using the fact that $\gamma$ is a $(1, \kappa)$-almost-geodesic of $ (\Omega, \koba_{\Omega} )$, and $\gamma|_{[0, s_0]}$ and $ \gamma|_{[t_0, a]}$ are $(1, \kappa_0)$-quasi-geodesics of $ (U \cap\Omega, \koba_{U \cap \Omega} )$, we have
\begin{equation*}
\begin{split}
    \koba_{\Omega}(z, w) 
    &\geq a - \kappa\\
    &\geq a - t_0 + s_0 -\kappa \quad [\text{since $s_0 \leq t_0$ }]\\
    &\geq \koba_{U \cap \Omega}(w, \gamma(t_0) ) - \kappa_0 + \koba_{U \cap \Omega}(\gamma(s_0), z) - \kappa_0  - \kappa\\ 
    &\geq \koba_{U \cap \Omega} (w, o) + \koba_{U \cap \Omega}(o, \gamma(t_0) ) - C_0 + \koba_{U \cap \Omega} (z, o) + \koba_{U \cap \Omega}(o, \gamma(s_0) ) - C_0 - 2 \kappa_0- \kappa\\
    &\geq \koba_{U \cap \Omega} (z, o) + \koba_{U \cap \Omega} (w, o)  - 2C_0- 2 \kappa_0- \kappa\\
    &\geq \koba_{U \cap \Omega} (z, w)  - 2C_0- 2 \kappa_0 - \kappa.
\end{split}    
\end{equation*}
The fourth inequality follows from (\ref{E:C_inequality_1}) and (\ref{E:C_inequality_2}). 
This proves the theorem by taking $C \defeq 2C_0 + 2 \kappa_0 + \kappa$.
\end{proof}

\begin{proof}[Proof of Theorem~\ref{T:Main_Thm_2}]
Let $W_1, W_2$ be open sets such that $W \subset \subset W_1 \subset \subset W_2 \subset \subset U$. Let $z, w \in W \cap \Omega$. Let $ \kappa > 0$ and $\gamma : [0, a] \longrightarrow \Omega$ be a $(1, \kappa)$-almost-geodesic of $(\Omega, \koba_{\Omega} )$ joining $z$ and $w$. 

\noindent
{\bf Case 1.} $\gamma \subset W_2 \cap \Omega $. 

Now, replacing $W$ by $W_2$ in Lemma~\ref{L:One_kappa_geodesic_Omega}, there exists $\kappa_0 > 0$ which depends on $W_2, U$ and $\kappa$ such that $\gamma$ is a $(1, \kappa_0)$-quasi-geodesic of $(U \cap \Omega, \koba_{U \cap \Omega} )$.

This gives, using the fact that $\gamma$ is a $(1, \kappa_0)$-quasi-geodesic of $(U \cap \Omega, \koba_{U \cap \Omega} )$ and $\gamma$ is a $(1, \kappa)$-almost-geodesic of $(\Omega, \koba_{\Omega} )$,
\[
\koba_{U \cap \Omega }(z, w) \leq a + \kappa_0 \leq \koba_{ \Omega }(z, w) + \kappa + \kappa_0.
\]
\noindent
{\bf Case 2.} $\gamma \nsubseteq W_0 \cap \Omega $.
Now, set
\[
s_0 = \sup \{t : t \in [0, a]\,\, \text{such that}\,\, \gamma(s) \in W_1 \,\forall s \in [0, t]\},
\]
\[
t_0= \inf \{t : t \in [0, a]\,\, \text{such that}\,\, \gamma(s) \in W_1 \,\forall s \in [t, a]\}.
\]

Note that by definition, $\gamma([0, s_0]) \subset W_2$ and $\gamma([t_0, a]) \subset W_2$. Hence $\sigma^1 \defeq \gamma|_{[0, s_0]}$ and $\sigma^2 \defeq \gamma|_{[t_0, a]}$ are $(1, \kappa_0)$-quasi-geodesic of $(U \cap \Omega, \koba_{U \cap \Omega} )$ (as observed in Case 1). Also note that $\sigma_1: [0, s_0] \longrightarrow \Omega$  and $\sigma_2: [t_0, a] \longrightarrow \Omega$ are $(1, \kappa)$-almost-geodesics of $(\Omega, \koba_{\Omega} )$ such that $\sigma_1(0), \sigma_2(a) \in W \cap \Omega$ and $\sigma_1(s_0), \sigma_2(t_0) \in \Omega \setminus W_1 $. Furthermore, $\overline W \cap \overline{ \Omega \setminus W_1} = \emptyset$. Now, since $(\Omega, \koba_{\Omega} )$ has  weak visibility for every pair of distinct points in $U \cap  \partial \Omega$, by {\it (2)} of Proposition ~\ref{P:Visibility_Outside_Point} after replacing $W$ and $V$ by $W$ and $\Omega \setminus W_1$ respectively, for $\kappa $ as above, there exists a compact set $K \subset  \Omega$  such that $\sigma^1 \cap K \neq \emptyset$ and $\sigma^2 \cap K \neq \emptyset$. Let $o^1 \in \sigma^1 \cap K \subset \overline W_1 \cap \Omega$ and $o^2 \in \sigma^2 \cap K \subset \overline  W_1 \cap \Omega$. Now, applying the fact that $\gamma$ is a $(1, \kappa)$-almost-geodesic of $(\Omega, \koba_{\Omega} )$ and $\sigma_1$ and $\sigma_2$ are $(1, \kappa_0)$-quasi-geodesic of $(U \cap \Omega, \koba_{U \cap \Omega} )$ and by the triangle inequality, we obtain the following sequence of inequalities:
\begin{equation*}
\begin{split}
    \koba_{\Omega}(z, w) 
    &\geq a - \kappa\\
    &\geq a - t_0 + s_0 -\kappa \quad [\text{since $s_0 \leq t_0$ }]\\
    &\geq \koba_{U \cap \Omega}(w, \gamma(t_0) ) - \kappa_0 + \koba_{U \cap \Omega}(\gamma(s_0), z) - \kappa_0  - \kappa\\ 
    &\geq \koba_{U \cap \Omega}(w, o^2) + \koba_{U \cap \Omega}(o^2,  \gamma(t_0) ) - 3\kappa_0 + \koba_{U \cap \Omega}(z, o^1) + \koba_{U \cap \Omega}(o^1, \gamma(s_0) ) - 3\kappa_0 - 2 \kappa_0 - \kappa\\
    &\geq \koba_{U \cap \Omega}(w, o^2) + \koba_{U \cap \Omega}(z, o^1)  - 8 \kappa_0 - \kappa\\
    &= \koba_{U \cap \Omega}(w, o^2) + \koba_{U \cap \Omega}(o^2, o^1) + \koba_{U \cap \Omega}(z, o^1)  - \koba_{U \cap \Omega}(o^2, o^1) - 8 \kappa_0- \kappa\\
    &\geq \koba_{U \cap \Omega}(z, w) - \sup_{x, y \in K \cap \overline W_1} \koba_{U \cap \Omega}(x, y) - 8\kappa_0 - \kappa.
\end{split}    
\end{equation*}
Note that $K \cap \overline  W_1$ is a relatively compact subset of $U \cap \Omega$, and hence $\sup_{x, y \in K \cap \overline  W_1} \koba_{U \cap \Omega}(x, y) < + \infty$. This proves the theorem by taking $C \defeq \sup_{x, y \in K \cap \overline  W_1} \koba_{U \cap \Omega}(x, y) + 8\kappa_0 + \kappa$.

\end{proof}
\begin{proof}[Proof of Corollary~\ref{C:Pseudo_finte_type}] 
This follows from Theorem~\ref{T:Main_Thm_2} by noting that $\Omega$ is  also a Goldilocks domain, by \cite{Bharali_Zimmer_1}. Hence $(\Omega, \koba_{\Omega} )$ has weak visibility property for every pair of distinct boundary points in $\partial \Omega$.
\end{proof}
Next, we give a proof of Theorem~\ref{T:Main_Thm_4} which is similar to the proof of \cite[Theorem]{NOT} by Nikolov-\"Okten-Thomas although our result holds in a more general setting.
\begin{proof}[Proof of Theorem~\ref{T:Main_Thm_4}]
We break the problem in two cases. Set $c\defeq \koba_{\Omega}(W_1 \cap \Omega, \Omega \setminus W_2) > 0$.
{\bf Case 1.}  $z, w \in W \cap \Omega$ such that $\koba_{\Omega} (z, w) \geq \frac{1}{4} \koba_{\Omega}(W_1 \cap \Omega, \Omega \setminus W_2) = \frac{c}{4}$. In this case, using the following (additive localization):
\[
\koba_{U \cap \Omega}(z, w) \leq \koba_{\Omega}(z, w) + C_0,
\] 
we get
\[
\frac{\koba_{U \cap \Omega}(z, w)}{\koba_{\Omega}(z, w)} \leq 1  + \frac{C_0}{\koba_{\Omega}(z, w)} \leq  1  + \frac{4C_0}{c}.
\]
{\bf Case 2.}  $z, w \in W \cap \Omega$ such that $\koba_{\Omega} (z, w) < c/4 $.
Suppose that $\gamma_n : [0, a_n] \longrightarrow \Omega$ is a sequence of $(1, c/4n)$-almost-geodesics such that $\gamma_n(0) = z$ and $\gamma_n (a_n) = w$ for all $n \in \posint$ (existence follows from \cite[Proposition~5.3]{Bharali_Zimmer_2}). 

\noindent {\bf Claim.} $\gamma_n \subset W_2$ for all $n \in \posint$.
\begin{proof}[Proof of the claim]
If  not then there exists $N \in \posint$ such that $\gamma_N \cap \Omega \setminus W_2 \neq \emptyset$. Let $z_0 \in \gamma_N \cap \Omega \setminus W_2 $. Since $z, w \in W \cap \Omega$ and $z_0 \in \Omega \setminus W_2$, it follows by definition that
\[
\koba_{\Omega}(z,z_0) \geq  \koba_{\Omega}(W_1 \cap \Omega, \Omega \setminus W_2)= c, \quad \koba_{\Omega}(w,z_0) \geq \koba_{\Omega}(W_1 \cap \Omega, \Omega \setminus W_2)= c.
\]
Since $\gamma_N$ is a $(1, c/4N)$-almost-geodesic of $(\Omega, \koba_{\Omega})$, using the inequalities above, we have 
\[
\koba_{\Omega}(z,w) \geq  \koba_{\Omega}(z,z_0) + \koba_{\Omega}(z_0,w) - 3c/4N  \geq  c + c - 3c/4N 
> c.
\]
This is a contradiction because by assumption $\koba_{\Omega}(z,w) < c/4$. Hence the claim is true.
\end{proof}
Now, using the Royden's Localization Lemma, it follows for all $z \in W_1 \cap \Omega$ and $v \in \C^d$,
\[
\kappa_{U \cap \Omega}(z;v) \leq c \kappa_{\Omega}(z;v).
\]
This gives (by definition $\gamma_n$ is is absolutely continuous, hence differentiable almost everywhere on $[0, a_n]$ for all $n \in \posint$) by \cite[Theorem~3.1]{Venturini}, for all $n \in \posint$,
\begin{equation}
\begin{split}
\koba_{U \cap \Omega}(z, w) 
& \leq \int_{0}^{a_n} \kappa_{U \cap \Omega}(\gamma_n(t); \gamma_n^{\prime}(t)) dt\\
& \leq c\int_{0}^{a_n} \kappa_{ \Omega}(\gamma_n(t); \gamma_n^{\prime}(t)) dt\\
&\leq  c (a_n - c/4n) +c^2/4n\\
& \leq c\koba_{\Omega}(z, w) + c^2/4n.
\end{split}
\end{equation}
By letting $n \to \infty$, we get 
\[
\koba_{U \cap \Omega}(z, w) \leq  c\koba_{\Omega}(z, w).
\]
By combining both cases, we obtain the required result.
\end{proof}
\begin{proof}[Proof of Theorem~\ref{T:Main_Thm_3}]
Since for the bounded domain $\Omega$, $\koba_{\Omega}(W_1 \cap U, \Omega \setminus W_2) > 0$ for every open sets $W_1, W_2$ with $W_1 \subset \subset W_2$ and $W_1 \cap \Omega \neq \emptyset$. This can be seen from \cite[Proposition 3.5.(1)]{Bharali_Zimmer_1}. Hence the result follows from Theorem~\ref{T:Main_Thm_4}.
\end{proof}
\noindent \textbf{Acknowledgements.} The author would like to thank Anwoy Maitra, Vikramjeet Singh Chandel and Sushil Gorai for useful  discussions and Gautam Bharali for useful suggestions and corrections. The author also would like to thank the referee for the corrections and many useful suggestions which improved the exposition of the previous version of this paper, and also for suggesting short and transparent proof of Lemma~\ref{L:Comple_dist_Compar}, and consequently, the proofs of the main theorems are simplified.

\end{document}